\documentclass[10pt,oneside,leqno]{amsart}
\usepackage{amsxtra}
\usepackage{amsopn}
\usepackage{color}
\usepackage{amsmath,amsthm,amssymb}
\usepackage{amscd}
\usepackage{amsfonts}
\usepackage{latexsym}
\usepackage{verbatim}
\usepackage{multirow}

\theoremstyle{plain}
\newtheorem{theorem}{Theorem}[section]
\newtheorem{definition}[theorem]{Definition}
\newtheorem{lemma}[theorem]{Lemma}
\newtheorem{proposition}[theorem]{Proposition}
\newtheorem{corollary}[theorem]{Corollary}
\newtheorem{remark}[theorem]{Remark}

\newtheorem{question}[theorem]{Question}
\newtheorem{remark-question}[section]{Remark-Question}

\newcommand\fra{{\mathfrak a}} 
\newcommand\frg{{\mathfrak g}}
\newcommand\frh{{\mathfrak h}}

\newcommand\frm{{\mathfrak m}}

\newcommand\Real{{\mathfrak R}{\frak e}\,} 
\newcommand\Imag{{\mathfrak I}{\frak m}\,}
\newcommand\nilm{\Gamma\backslash G}

\newcommand\db{{\bar{\partial}}}

\begin{document}
\title[]{
A family of complex nilmanifolds with infinitely many\\ real homotopy types
}


\author{Adela Latorre}
\address[A. Latorre and R. Villacampa]{Centro Universitario de la Defensa\,-\,I.U.M.A., Academia General
Mili\-tar, Crta. de Huesca s/n. 50090 Zaragoza, Spain}
\email{adela@unizar.es}
\email{raquelvg@unizar.es}

\author{Luis Ugarte}
\address[L. Ugarte]{Departamento de Matem\'aticas\,-\,I.U.M.A.\\
Universidad de Zaragoza\\
Campus Plaza San Francisco\\
50009 Zaragoza, Spain}
\email{ugarte@unizar.es}

\author{Raquel Villacampa}


\maketitle

\begin{abstract}
We find a one-parameter family of non-isomorphic nilpotent Lie algebras $\frg_a$, with $a \in [0,\infty)$,
of real dimension eight with (strongly non-nilpotent) complex structures. By restricting~$a$ to take rational values, we arrive
at the existence of infinitely many real homotopy types of $8$-dimensional nilmanifolds admitting a complex structure.
Moreover, balanced Hermitian metrics and generalized Gauduchon metrics on such nilmanifolds are constructed.
\end{abstract}

\maketitle


\section{Introduction}\label{intro}

\noindent
Let $\frg$ be an even-dimensional real nilpotent Lie algebra.  A complex structure on $\frg$ is an endomorphism $J\colon\frg\longrightarrow\frg$
satisfying $J^2=-\textrm{Id}$ and the integrability condition given by the vanishing of the Nijenhuis tensor, i.e.
\begin{equation}\label{Nijenhuis}
N_J(X,Y):=[X,Y]+J[JX,Y]+J[X,JY]-[JX,JY]=0,
\end{equation}
for all $ X,Y\in\frg$.
The classification of nilpotent Lie algebras endowed with such structures has interesting geometrical applications; for instance,
it allows to construct complex nilmanifolds and study their geometric properties.
Let us recall that a nilmanifold is a compact quotient $\nilm$ of a connected, simply connected,
nilpotent Lie group $G$ by a lattice $\Gamma$ of maximal rank in $G$.
If the Lie algebra $\frg$ of $G$ has a complex structure~$J$, then a compact complex manifold $X=(\nilm,J)$
is defined in a natural way.

The problem of determining which nilpotent Lie algebras admit a complex structure is completely solved only up to dimension 6.
In real dimension 4 there are two nilpotent Lie algebras with a complex structure, namely, the abelian Lie algebra and
the Lie algebra underlying the Kodaira-Thurston manifold.
The classification of $6$-dimensional nilpotent Lie algebras having a complex structure was achieved by Salamon in~\cite{Salamon},
and there are precisely 18 isomorphism classes.
In higher dimensions only partial results are known (see~\cite{LUV-SnN} and the references therein).
In particular, note that there is no classification of real nilpotent Lie algebras of dimension~8.

In \cite{DF2000,DF2003} Dotti and Fino study the $8$-dimensional nilpotent Lie algebras $\frg$ that admit a
\emph{hypercomplex} structure, i.e. a pair of anticommuting complex structures $\{J_i\}_{i=1,2}$.
They prove that if $\frg$ admits such a structure, then $\frg$ must be 2-step
nilpotent and have first Betti number $b_1(\frg) \geq 4$.
They also find an explicit description of the $8$-dimensional nilpotent Lie
algebras admitting a hypercomplex structure in terms of families which depend on several real parameters.
However, it seems to be unclear whether these families contain an infinite number of
pairwise non-isomorphic Lie algebras or not.
Motivated by this fact, we address the following more general problem:
\begin{question}\label{Q}
Do there exist infinite non-isomorphic nilpotent Lie algebras in dimension 8 admitting complex structures?
\end{question}

In this paper we provide an affirmative answer to this question. In addition, we show
that there are infinitely many real homotopy types of $8$-dimensional nilmanifolds admitting a complex structure.
Moreover, the nilmanifolds that we construct can be endowed with both generalized Gauduchon and balanced Hermitian metrics.
It should be noted that our results are based on the theory of strongly non-nilpotent complex structures developed
in~\cite{LUV-SnN}.

The paper is structured as follows. In Section~\ref{sect2} we review the main results on complex structures found in~\cite{LUV-SnN}.
It turns out that the essentially new complex structures on nilpotent Lie algebras that arise in each even real dimension are those of strongly
non-nilpotent type. A complex structure~$J$ on a nilpotent Lie algebra $\frg$ is said to be \emph{strongly
non-nilpotent} (SnN for short) if the center of $\frg$ does not contain any non-trivial $J$-invariant ideal.
In~\cite[Theorem 4.1]{LUV-SnN} a structure result is obtained for the $8$-dimensional nilpotent Lie algebras admitting
an SnN complex structure in terms of the dimensions of their ascending central series.
Furthermore, a complete description of the SnN complex geometry is given (see~\cite[Propositions 4.12, 4.13, and 4.14]{LUV-SnN}).

We make use of~\cite[Proposition 4.12]{LUV-SnN} to answer Question~\ref{Q}. Such result describes the generic complex
equations of any SnN complex structure on $8$-dimensional nilpotent Lie algebras $\frg$ with ascending central series
$\{\frg_k\}_k$ of dimensions $(1, 3, 8)$, $(1, 3, 5, 8)$,
$(1, 3, 6, 8)$, or $(1, 3, 5, 6, 8)$
(see Proposition~\ref{complexification_(1,3,5,8)+(1,3,5,6,8)}).
In Section~\ref{sect3} we present a specific choice of complex parameters that allows us to construct
a family $\frg_{a}$, $a \in [0,\infty)$, of pairwise non-isomorphic $8$-dimensional
nilpotent Lie algebras endowed with complex structures
(see Theorem~\ref{teorema-main}).
Note that the ascending central series of each $\frg_a$ is of type $(1, 3, 5, 8)$, thus
$\frg_a$ is 4-step nilpotent and has first Betti number $b_1(\frg_a)=3$.

In Section~\ref{nilvariedades} we consider the family
$\frg_a$ with rational values of the parameter $a$ in order to show that in eight dimensions there are
infinitely many \emph{real} homotopy types of nilmanifolds admitting a complex structure (see Theorem~\ref{infinite}).
Notice that eight is the lowest dimension where this can occur, since for any even dimension less than or equal to 6,
only a finite number of real homotopy types of nilmanifolds exists.
Indeed, in six dimensions, Bazzoni and Mu\~noz prove in~\cite[Theorem~2]{BM2012}
that there are infinitely many \emph{rational} homotopy types of nilmanifolds,
but only 34 different real homotopy types.
We also compute the de Rham cohomology of the given nilmanifolds $N_a$, which allows us to
show that $N_a$ does not admit any symplectic structure for any $a$. 

The last section is devoted to study the existence of special Hermitian metrics on the nilmanifolds~$N_{a}$ endowed with
the strongly non-nilpotent complex structure~$J_{a,1}$ found in Section~\ref{sect3}.
We prove in Theorem~\ref{infinite-1-Gauduchon} that the complex nilmanifolds $X_a=(N_a,J_{a,1})$
have Hermitian metrics satisfying the $k$-th Gauduchon condition for every $k$~\cite{FWW}.
Moreover, $X_a$ also admits balanced~\cite{Mi} (hence, strongly Gauduchon~\cite{Pop0,Pop1}) Hermitian metrics
(see Theorem~\ref{infinite-balanced}).
Therefore, there are infinitely many real homotopy types of $8$-dimensional nilmanifolds
with both generalized Gauduchon and balanced metrics.

\medskip

\noindent{\textbf{Acknowledgments}}.
We are grateful to A. Fino for informing us about Question~\ref{Q}
and suggesting us to investigate it in
the class of strongly non-nilpotent complex structures.
This work has been partially supported by the projects Mineco (Spain) MTM2014-58616-P, and
Gobierno de Arag\'on/Fondo Social Europeo--Grupo Consolidado E15 Geometr\'{\i}a.

\section{Complex structures on nilpotent Lie algebras}\label{sect2}

\noindent In this section we collect some known results about real nilpotent Lie algebras (NLA for short)
endowed with complex structures. In particular, we pay special attention to real dimension~$8$
when the complex structure is of strongly non-nilpotent type, recalling the main ideas in~\cite{LUV-SnN}.

A well-known invariant of a Lie algebra $\frg$ is its \emph{ascending central series} $\{\frg_k\}_k$, whose terms are given by
\begin{equation}\label{ascending-series}
\left\{\begin{array}{l}
\frg_0=\{0\}, \text{ and } \\[4pt]
\frg_k=\{X\in\frg \mid [X,\frg]\subseteq \frg_{k-1}\}, \text{ for } k\geq 1.
\end{array}\right.
\end{equation}
Notice that $\frg_1=Z(\frg)$ is the center of~$\frg$.
A Lie algebra $\frg$ is \emph{nilpotent} if there is
an integer~$s\geq 1$ such that~$\frg_k=\frg$, for every~$k\geq s$.
In such case, the smallest integer $s$ satisfying the condition is called the
\emph{nilpotency step} of~$\frg$, and the Lie algebra is said to be \emph{$s$-step nilpotent}.

Let $J$ be a complex structure on an NLA $\frg$, that is,
an endomorphism $J\colon\frg\longrightarrow\frg$
fulfilling $J^2=-\textrm{Id}$ and the integrability condition~\eqref{Nijenhuis}.
Observe that the terms~$\frg_{k}$ in the series~\eqref{ascending-series}
may not be invariant under~$J$. For this reason, a new series~$\{\fra_{k}(J)\}_{k}$ adapted to the complex structure~$J$
is introduced in~\cite{CFGU-dolbeault}:
\begin{equation}\label{adapted-series}
\left\{\begin{array}{l}
\fra_0(J)=\{0\}, \text{ and } \\[4pt]
\fra_k(J)=\{X\in\frg \mid [X,\frg]\subseteq \fra_{k-1}(J)\ {\rm and\ } [JX,\frg]\subseteq \fra_{k-1}(J)\}, \text{ for } k\geq 1.
\end{array}\right.
\end{equation}
This series $\{\fra_{k}(J)\}_{k}$
is called the \emph{ascending $J$-compatible series of~$\frg$}.
Observe that every $\fra_k(J)\subseteq\frg_{k}$ is an even-dimensional $J$-invariant ideal of $\frg$, and
$\fra_1(J)$ is indeed the largest subspace of the center $\frg_1$ which is $J$-invariant.

Depending on the
behaviour of the series~$\{\fra_{k}(J)\}_{k}$,
complex structures on NLAs can be classified into different types:


\begin{definition}\label{tipos_J}
\cite{CFGU-dolbeault, LUV-SnN}
{\rm
A complex structure $J$ on a nilpotent Lie algebra $\frg$ is said to be
\begin{itemize}
\item[(i)] \emph{strongly non-nilpotent}, or \emph{SnN} for short, if $\fra_1(J)=\{0\}$;

\smallskip
\item[(ii)] \emph{quasi-nilpotent}, if $\fra_1(J)\neq\{0\}$; moreover, $J$ is called
 \begin{itemize}
 \item[(ii.1)] \emph{nilpotent}, if there exists an integer~$t>0$ such that~$\fra_t(J)=\frg$,
 \item[(ii.2)] \emph{weakly non-nilpotent}, if there is an integer~$t>0$ satisfying~$\fra_t(J)=\fra_l(J)$, for every~$l\geq t$,
           and~$\fra_t(J)\neq\frg$.
 \end{itemize}
\end{itemize}
}
\end{definition}

One can see that quasi-nilpotent complex structures on
NLAs of a given dimension can be constructed from
other complex structures defined on (strictly) lower dimensional NLAs (see \cite[Section~2]{LUV-SnN} for details).
Therefore, the essentially new complex structures that arise in each even real dimension are those of
strongly non-nilpotent type. That is to say, SnN complex structures constitute the remaining piece to completely understand
complex geometry on nilpotent Lie algebras.

In real dimension $4$ it is well known that SnN complex structures do not exist, whereas
in dimension $6$ one has the following result:

\begin{theorem}\label{structural-dim6}\cite{U,UV1}
Let $\frg$ be an NLA of real dimension $6$. If $\frg$ admits an SnN complex structure, then
the terms of its ascending central series $\{\frg_k\}_k$ have dimensions
$(\dim\frg_k)_k=$~$(1,3,6)$ or $(1,3,4,6)$.
\end{theorem}

In fact, all the pairs $(\frg,J)$ with $\text{dim}\,\frg=6$ and $\fra_1(J)=\{0\}$
have been classified by means of their complex structure equations.
It should be noted that only two NLAs of this dimension admit SnN complex structures,
namely, $\frh_{19}^-$ and $\frh_{26}^+$ in the notation of~\cite{U,UV1}.

Concerning higher dimensions, \cite{LUV-SnN} provides
several general restrictions on the terms of the ascending central series of NLAs admitting
SnN complex structures.
Among them, we highlight the following one:

\begin{theorem}\label{prop_centro}\cite[Theorem 3.11]{LUV-SnN}
Let $(\frg, J)$ be a $2n$-dimensional nilpotent Lie algebra, with $n\geq 4$, endowed with a strongly
non-nilpotent complex structure $J$. Then, $1\leq \dim\frg_1 \leq n-3$.
\end{theorem}

Thanks to this result and using the \emph{doubly adapted basis method} developed in~\cite{LUV-SnN},
a structural result in the spirit of Theorem~\ref{structural-dim6} is proved for dimension $8$:

\begin{theorem}\label{teorema-estructura-acs}\cite[Theorem 4.1]{LUV-SnN}
Let $\frg$ be an NLA of real dimension $8$. If $\frg$ admits an SnN complex structure, then
the terms of its ascending central series $\{\frg_k\}_k$ have dimensions
$(\dim\frg_k)_k=$~$(1,3,8)$, $(1,3,5,8)$, $(1,3,6,8)$, $(1,3,5,6,8)$, $(1,4,8)$, $(1,4,6,8)$, $(1,5,8)$, or $(1,5,6,8)$.
\end{theorem}

Moreover, the complex structure equations of all the previous pairs $(\frg,J)$ with $\dim\frg=8$ and $\fra_1(J)=\{0\}$
are parametrized (see~\cite[Propositions 4.12, 4.13, and 4.14]{LUV-SnN}). For the aim of this paper,
we focus on the proposition:

\begin{proposition}\label{complexification_(1,3,5,8)+(1,3,5,6,8)}
{\rm \cite[Proposition 4.12]{LUV-SnN}}
Let $J$ be a strongly non-nilpotent complex structure on an 8-dimensional nilpotent
Lie algebra $\frg$ such that $\dim \frg_1=1$ and $\dim \frg_2=3$.
Then, there is a complex basis $\{\omega^1,\omega^2,\omega^3,\omega^4\}$ 
of bidegree $(1,0)$ with respect to $J$ satisfying the structure equations\footnote{\textsc{Notation:}
$\omega^{ij}=\omega^i\wedge\omega^j$ and $\omega^{i\bar k}=\omega^i\wedge\omega^{\bar k}$, where
$\omega^{\bar k}$ is the conjugate of $\omega^k$.}
\begin{equation*}
\left\{\begin{array}{rcl}
d\omega^1 &\!\!\!=\!\!\!& 0,\\[3pt]
d\omega^2 &\!\!\!=\!\!\!& A\,\omega^{1\bar 1} -B(\omega^{14}- \omega^{1\bar 4}),\\[4pt]
d\omega^3 &\!\!\!=\!\!\!& (C-D)\,\omega^{12}-E\,(\omega^{14}- \omega^{1\bar 4})+F\,\omega^{1\bar 1}+(G+D)\,\omega^{1\bar 2}\\[3pt]
 && -H\,(\omega^{24}-\omega^{2\bar 4}) +(C-G)\,\omega^{2\bar1}+K\,\omega^{2\bar 2},\\[4pt]
d\omega^4 &\!\!\!=\!\!\!& L\,\omega^{1\bar 1}+M\,\omega^{1\bar 2}+N\,\omega^{1\bar 3}-\bar{M}\,\omega^{2\bar 1}
  +i\,s\,\omega^{2\bar 2}+P\,\omega^{2\bar 3}-\bar{N}\,\omega^{3\bar 1}-\bar{P}\,\omega^{3\bar2},\end{array}\right.
\end{equation*}
for some coefficients $s\in\mathbb R$ and $A,B,C,D,E,F,G,H,K,L,M,N,P\in\mathbb C$.
\end{proposition}

Notice that the coefficients above must fulfill the (non-linear) equations that guarantee the Jacobi identity of the Lie algebra,
i.e. $d(d \omega^k)=0$ for $1\leq k\leq 4$.

\section{A family of non-isomorphic 8-dimensional nilpotent Lie algebras with complex structures}\label{sect3}

\noindent
In this section we find an infinite family of (non-isomorphic) $8$-dimensional nilpotent Lie algebras admitting complex structures,
providing an affirmative answer to Question~\ref{Q}.

In the complex equations given in Proposition~\ref{complexification_(1,3,5,8)+(1,3,5,6,8)},
we choose the following particular values of the parameters:

\vskip.2cm

\hskip2.5cm $A = E = K = L = P = 0,\quad \ B=N=s=1$,

\vskip.25cm

\hskip2.5cm $
C = \frac i2,  \quad
D = \frac{3\,i}{2},  \quad
F = a, \quad
G = -\frac i2, \quad
H = -i, \quad M = i\,b$,

\vskip.2cm

\noindent where $a,b\in\mathbb R$. That is to say, we consider the complex structure equations
\begin{equation}\label{eleccion}
\left\{\begin{array}{rcl}
d\omega^1 &\!\!\!=\!\!\!& 0,\\[3pt]
d\omega^2 &\!\!\!=\!\!\!& -\omega^{14} + \omega^{1\bar 4},\\[4pt]
d\omega^3 &\!\!\!=\!\!\!& a\,\omega^{1\bar 1} - i\,(\omega^{12} - \omega^{1\bar 2} - \omega^{2\bar1})
  + i\,(\omega^{24}-\omega^{2\bar 4}),\\[4pt]
d\omega^4 &\!\!\!=\!\!\!& i\,b\,(\omega^{1\bar 2} + \omega^{2\bar 1}) + \omega^{1\bar 3}
  +i\,\omega^{2\bar 2}-\omega^{3\bar 1}.
  \end{array}\right.
\end{equation}
It is easy to see that the Jacobi identity holds, so for
each pair $(a,b)\in\mathbb R^2$,
these structure equations define a nilpotent Lie algebra of real dimension 8
endowed with a complex structure. We will study the underlying real nilpotent Lie algebras and show that there is an
infinite number of them.

Let $\{e^i\}_{i=1}^8$ be the real basis determined by
\begin{equation*}
\begin{array}{lllll}
& e^1=-2\,\Real\omega^1, \ \ & e^2=2\,\Imag\omega^1, \ \ & e^3=-2\,\Real\omega^4, \ \ & e^4=2\,\Real\omega^2, \\[5pt]
& e^5=-2\,\Imag\omega^2, \ \ & e^6=-4\,\Real\omega^3, \ \ & e^7=4\,\Imag\omega^3, \ \ & e^8=4\,\Imag\omega^4.
\end{array}
\end{equation*}
Equivalently,
$$
\omega^1=\frac 12\,(-e^1+i\,e^2), \quad
\omega^2=\frac 12\,(e^4-i\,e^5), \quad
\omega^3=\frac 14\,(-e^6+i\,e^7), \quad
\omega^4=\frac 14\,(-2\,e^3+i\,e^8).
$$
Then, for any $(a,b)\in\mathbb R^2$, one has a real nilpotent Lie algebra $\frg_{a,b}$ defined by the following
structure equations coming from~\eqref{eleccion}:
\begin{equation}\label{ec-reales-FII}
\frg_{a,b}: \quad
\left\{\begin{split}
de^1 &= de^2 = de^3= 0,\\
de^4 &= e^{13},\\
de^5 &= e^{23},\\
de^6 &=3\,e^{14} + \,e^{25} - 2\,e^{35},\\
de^7 &= 2\,a\,e^{12} + e^{15}+e^{24} + 2\,e^{34},\\
de^8 &= - 2b\,e^{14} + e^{16} - 2b\, e^{25} + e^{27} - 2\,e^{45},
\end{split}\right.
\end{equation}
where $e^{ij}=e^i\wedge e^j$.
The integrable almost complex structure on $\frg_{a,b}$ is expressed in this real basis~as
\begin{equation}\label{J-a,b}
J_{a,b}(e^1)=e^2,\quad J_{a,b}(e^3)=\frac{1}{2}e^8,\quad J_{a,b}(e^4)=e^5,\quad J_{a,b}(e^6)=e^7.
\end{equation}

In the next lines, we will prove the following result:

\begin{proposition}\label{teorema-no-iso-NLA-bis}
If the nilpotent Lie algebras $\frg_{a,b}$ and $\frg_{a',b'}$ are isomorphic,
then there is a non-zero real number $\rho$ such that
$$
a' = \pm\, \rho\, a, \quad\mbox{ and }\quad b' = \rho\, b.
$$
\end{proposition}

As a consequence, we are able to provide a positive answer to Question~\ref{Q}:

\begin{theorem}\label{teorema-main}
The real Lie algebras $\frg_{a,1}$, $a \in [0,\infty)$, define a family of pairwise non-isomorphic
nilpotent Lie algebras admitting complex structures.
\end{theorem}

\begin{proof}
Using Proposition~\ref{teorema-no-iso-NLA-bis} for $b'=b=1$ and for non-negative values of $a$ and $a'$,
we see that
if $\frg_{a,1}$ is isomorphic to $\frg_{a',1}$ then
$\rho=1$ and $a' = a$. Hence, different values of $a$ in $[0,\infty)$ produce non-isomorphic Lie algebras $\frg_{a,1}$.
Since $J_{a,1}$ in \eqref{J-a,b} is an integrable almost complex structure on $\frg_{a,1}$,
any nilpotent Lie algebra in this family
has complex structures.
\end{proof}

\begin{remark}\label{clasif-g-a,b}
{\rm
Any Lie algebra $\frg_{a,b}$ is isomorphic to one and only one of the following: $\frg_{0,0}$, $\frg_{1,0}$,
or $\frg_{a,1}$, with $a \in [0,\infty)$.
Indeed:
\begin{itemize}
\item If $a\neq 0$ and $b=0$, then $\frg_{a,0}$ is isomorphic to $\frg_{1,0}$. In fact, it is enough to
apply to the equations~\eqref{ec-reales-FII} the isomorphism given by
${\rm diag}\,(a^{-1},a^{-1},a^{-1},a^{-2},a^{-2},a^{-3},a^{-3},a^{-4}) \in {\rm GL}(8,\mathbb{R})$.
\item If $b\neq 0$, then the Lie algebra $\frg_{a,b}$ is isomorphic to $\frg_{|\frac{a}{b}|,1}$.
To see this, we first
apply to the equations~\eqref{ec-reales-FII} the transformation defined by
${\rm diag}\,(b^{-1},b^{-1},b^{-1},b^{-2},b^{-2},b^{-3},b^{-3},b^{-4}) \in {\rm GL}(8,\mathbb{R})$ in order to get an isomorphism from
$\frg_{a,b}$ to $\frg_{\frac{a}{b},1}$
Now, ${\rm diag}\,(1,-1,-1,-1,1,-1,1,-1)$ gives an isomorphism
between the Lie algebras $\frg_{\frac{a}{b},1}$ and $\frg_{-\frac{a}{b},1}$.
\item Finally, applying Proposition~\ref{teorema-no-iso-NLA-bis} to the Lie algebras $\frg_{0,0}$, $\frg_{1,0}$, and $\frg_{a,1}$, $a \in [0,\infty)$,
one concludes that any two of them are never isomorphic.
\end{itemize}
}
\end{remark}

\medskip

\noindent{\textbf{Proof of Proposition~\ref{teorema-no-iso-NLA-bis}}.}

\medskip

\noindent The rest of this section is devoted to the proof of Proposition~\ref{teorema-no-iso-NLA-bis}.
First, we normalize several coefficients in the structure equations~\eqref{ec-reales-FII}
in order to get simpler equations for the Lie algebras.
With this aim,
let us introduce the new basis $\{v^i\}_{i=1}^8$ given by
$$
\begin{aligned}
v^1 &= e^1, &
v^2 &= \frac{1}{\sqrt 3}\,e^2, &
v^3 &=\frac{2}{\sqrt 3}\,e^3, &
v^4 &=\frac{2}{\sqrt 3}\,e^4, \\[5pt]
v^5 &= \frac{2}{3}\,e^5, &
v^6 &= \frac{2}{3\,\sqrt 3}\,e^6, &
v^7 &= \frac{2}{3}\,e^7, &
v^8 &= \frac{2}{3\,\sqrt 3}\,e^8 + \frac{4\,b}{9\,\sqrt 3}\,e^6.
\end{aligned}
$$
In terms of this basis, the structure equations~\eqref{ec-reales-FII} become
\begin{equation}\label{ec-reales-FII-caseiii}
\left\{\begin{split}
dv^1 &= dv^2 =dv^3= 0,\\
dv^4 &= v^{13},\\
dv^5 &= v^{23},\\
dv^6 &= v^{14}+\,v^{25}-v^{35},\\
dv^7 &= \alpha\,v^{12}+v^{15}+v^{24}+v^{34},\\
dv^8 &= v^{16}-2\beta\,v^{25}+v^{27}-\beta\,v^{35}-v^{45},
\end{split}\right.
\end{equation}
where $\alpha = \frac{4\,a}{\sqrt3}$ and $\beta= \frac{2\,b}{3}$. From now on, we will denote the real Lie algebra
by $\frm_{\alpha,\beta}$
(instead of~$\frg_{a,b}$) to indicate that we are considering the structure equations~\eqref{ec-reales-FII-caseiii} above.

We first observe that the algebras $\mathfrak m_{\alpha,\beta}$ share the same values for the usual invariants
of nilpotent Lie algebras. More precisely,
it can be seen that the dimensions of the terms in the ascending central series are~$(1,3,5,8)$,
and those of the descending central series are~$(8,5,3,1)$. Furthermore,
the Lie algebras cannot be distinguished by their cohomology groups
(see Section~\ref{nilvariedades} for more details).
Hence, we are led to directly study the existence of an isomorphism between
the real Lie algebras $\mathfrak m_{\alpha,\beta}$ in this family.

Let~$\mathfrak m_{\alpha,\beta}$ and~$\mathfrak m_{\alpha',\beta'}$
be two nilpotent Lie algebras defined by the structure equations~\eqref{ec-reales-FII-caseiii}
for $(\alpha,\beta)$ and $(\alpha',\beta')$, respectively.
Let $f \colon \mathfrak m_{\alpha,\beta} \longrightarrow \mathfrak m_{\alpha',\beta'}$ be a homomorphism of Lie algebras.
Hence, the dual map $f^* \colon \mathfrak m_{\alpha',\beta'}^* \longrightarrow \mathfrak m_{\alpha,\beta}^*$ extends to a map
$F \colon \bigwedge^*\mathfrak m_{\alpha',\beta'}^* \longrightarrow \bigwedge^*\mathfrak m_{\alpha,\beta}^*$ that commutes with the differentials, i.e. $F\circ d=d\circ F$.

Suppose that there exists a Lie algebra isomorphism between $\mathfrak m_{\alpha,\beta}$ and~$\mathfrak m_{\alpha',\beta'}$, and
let $\{v^k\}_{k=1}^8$ (resp. $\{v'^{\,k}\}_{k=1}^8$) be a basis for~$\mathfrak m_{\alpha,\beta}^*$
(resp. $\mathfrak m_{\alpha',\beta'}^*$) satisfying equations~\eqref{ec-reales-FII-caseiii}
with $(\alpha,\beta)$ (resp. $(\alpha',\beta')$).
In terms of these bases, any Lie algebra isomorphism is defined by
\begin{equation}\label{cambio-base}
F(v'^{\,i}) =\sum_{j=1}^8 \lambda_j^i\, v^j, \quad i=1,\ldots,8,
\end{equation}
satisfying the conditions
\begin{equation}\label{condiciones}
F(dv^i) = d(F(v'^{\,i})), \ \text{ for each }1\leq i\leq 8,
\end{equation}
where the matrix $\Lambda=(\lambda^i_j)_{i,j=1,\ldots,8}$ belongs to ${\rm GL}(8,\mathbb{R})$.

\medskip

In what follows, we will prove that $\lambda^1_1 \not= 0$, and
$$\alpha' = \pm\, \lambda^1_1\, \alpha, \quad\mbox{  }\quad \beta' = \lambda^1_1\, \beta.$$
Notice that this result implies Proposition~\ref{teorema-no-iso-NLA-bis}, since $a,b,a',b'$ are related to
$\alpha,\beta,\alpha',\beta'$ by $\alpha = \frac{4\,a}{\sqrt3}$, $\beta= \frac{2\,b}{3}$, $\alpha' = \frac{4\,a'}{\sqrt3}$ and $\beta'= \frac{2\,b'}{3}$.


\begin{lemma}\label{lema1}
The elements $F(v'^{\,i}) \in \mathfrak m_{\alpha,\beta}^*$ satisfy the following conditions:
\begin{equation*}
\begin{array}{rl}
& F(v'^{\,i}) \wedge v^{123} =0, \ \ \text{ for } \ i=1,2,3;\\[4pt]
& F(v'^{\,i}) \wedge v^{12345} =0,\ \ \text{ for } \ i=4,5;\\[4pt]
& F(v'^{\,i}) \wedge v^{1234567} =0,\ \ \text{ for } \ i=6,7.
\end{array}
\end{equation*}
In particular, the matrix $\Lambda=(\lambda^i_j)_{i,j} \in {\rm GL}(8,\mathbb{R})$ that determines $F$
is a block triangular matrix, and thus
$$\det(\lambda^i_j)_{i,j=1,2,3} \cdot
  \det(\lambda^i_j)_{i,j=4,5} \cdot
  \det(\lambda^i_j)_{i,j=6,7}\cdot \lambda^8_8 = \det \Lambda \,\neq\, 0.$$
\end{lemma}

\begin{proof}
First, we observe that from the equations~\eqref{ec-reales-FII-caseiii}, the equalities \eqref{condiciones} for $i=1,2,3$
simply read as $d(F(v'^{\,i}))=0$. Performing this calculation, one immediately gets
\begin{equation}\label{m3-reales1}
\lambda^i_j=0, \text{ for } 1\leq i\leq 3 \text{ and } 4\leq j\leq 8,
\end{equation}
and thus we have $F(v'^{\,i})\in \langle v^1,v^2,v^3\rangle$, or equivalently $F(v'^{\,i}) \wedge v^{123} =0$, for $i=1,2,3$.

Bearing in mind the values \eqref{m3-reales1}, one can see that \eqref{condiciones} for $i=4,5$ can be written as
\begin{equation}\label{exxpr}
\begin{split}
0 = d(F(v'^{\,i})) - F(dv'^{\,i}) &= (\lambda^{i-3}_2\,\lambda^3_1 - \lambda^{i-3}_1\,\lambda^3_2 + \alpha\,\lambda^i_7)\,v^{12}
+ (\lambda^{i-3}_3\,\lambda^3_1 - \lambda^{i-3}_1\,\lambda^3_3 + \lambda^i_4)\,v^{13} \\[2pt]
&
\ \ + \lambda^i_6\,v^{14} + \lambda^i_7\,v^{15}
+ \lambda^i_8\,v^{16}
+ (\lambda^{i-3}_3\,\lambda^3_2 - \lambda^{i-3}_2\,\lambda^3_3 + \lambda^i_5)\,v^{23}
+ \lambda^i_7\,v^{24} \\[2pt]
&
\ \ + (\lambda^i_6 - 2\,\beta\,\lambda^i_8)\,v^{25}
+ \lambda^i_8\,v^{27} + \lambda^i_7\,v^{34}
- (\lambda^i_6 + \beta\,\lambda^i_8)\,v^{35}
+ \lambda^i_8\,v^{45}.
\end{split}
\end{equation}
As a consequence, the vanishing of the coefficients in $v^{13}$, $v^{14}$, $v^{15}$, $v^{16}$, and $v^{23}$, implies that
$\lambda^i_6 = \lambda^i_7 = \lambda^i_8 = 0$,
and
\begin{equation}\label{m3-reales4}
\lambda^i_4 = \lambda^{i-3}_1\,\lambda^3_3 - \lambda^{i-3}_3\,\lambda^3_1, \qquad
\lambda^i_5 = \lambda^{i-3}_2\,\lambda^3_3 - \lambda^{i-3}_3\,\lambda^3_2,\qquad \mbox{ for } i=4,5.
\end{equation}
Therefore, $F(v'^{\,i})\in \langle v^1,\ldots,v^5\rangle$, for $i=4,5$, as desired.

Let us now consider $i=6,7$ in the equalities \eqref{condiciones}. It is easy to check that
$$
d(F(v'^{\,i})) - F(dv'^{\,i}) = \lambda^i_8\,v^{16} + \xi_i,
$$
where $\xi_i \in \langle v^{16} \rangle^{\perp} \subset \bigwedge^2\langle v^1,\ldots,v^8\rangle$.
Here, and in what follows, the orthogonal is taken with respect to the inner product in $\bigwedge^2\langle v^1,\ldots,v^8\rangle$
defined by declaring the standard basis $\{v^{i\,j}\}_{1 \leq i< j \leq 6}$ to be orthonormal.
Therefore, the equalities \eqref{condiciones} for $i=6,7$ in particular imply
\begin{equation*}
\lambda^6_8 = \lambda^7_8 = 0.
\end{equation*}
Hence, $F(v'^{\,i})\in\langle v^1,\ldots,v^7\rangle$, for $i=6,7$, and the proof of the lemma is complete.
\end{proof}

We now make use of Lemma~\ref{lema1} to prove the following sharper result.

\begin{lemma}\label{lema2}
The elements $F(v'^{\,i}) \in \mathfrak m_{\alpha,\beta}^*$
additionally satisfy $F(v'^{\,i}) \wedge v^{i} =0$, for $i=1,2,3$,
and
$$
F(v'^{\,4}) \wedge v^{1234} =0,\quad
F(v'^{\,5}) \wedge v^{1235} =0,\quad
F(v'^{\,6}) \wedge v^{123456} =0,\quad \text{ and } \quad
F(v'^{\,7}) \wedge v^{123457} =0.
$$
Therefore, $\Lambda=(\lambda^i_j)_{i,j} \in {\rm GL}(8,\mathbb{R})$
is a lower triangular matrix, and
$$
\prod_{i=1}^8 \lambda^i_i = \det \Lambda \,\neq\, 0.
$$
\end{lemma}

\begin{proof}
Let us first recall that all the conditions in Lemma~\ref{lema1} and those included along its proof must be satisfied in order to
ensure that $F$ is an isomorphism of Lie algebras.

A direct calculation shows that
\begin{equation*}
\begin{split}
d(F(v'^{\,8})) - F(dv'^{\,8}) =& - (\lambda^1_3\,\lambda^6_6 + \lambda^2_3\,\lambda^7_6)\,v^{36}
- (\lambda^1_3\,\lambda^6_7 + \lambda^2_3\,\lambda^7_7)\,v^{37} + \xi,
\end{split}
\end{equation*}
where $\xi \in \langle v^{36},v^{37} \rangle^{\perp} \subset \bigwedge^2\langle v^1,\ldots,v^8\rangle$.
Now, the condition \eqref{condiciones} for $i=8$ implies the
vanishing of the coefficients in $v^{36}$ and $v^{37}$ above, so the following homogeneous system
must be satisfied:
$$\begin{pmatrix}
\lambda^6_6 &\lambda^7_6\\ \lambda^6_7&\lambda^7_7
\end{pmatrix}\begin{pmatrix}
\lambda^1_3\\ \lambda^2_3
\end{pmatrix} = \begin{pmatrix}
0\\ 0
\end{pmatrix}.$$
Since by Lemma~\ref{lema1} one has $\det(\lambda^i_j)_{i,j=6,7}\not=0$, we conclude that $\lambda^1_3=\lambda^2_3=0$.
Therefore,
\begin{equation}\label{lema2-F2}
F(v'^{\,i})\in\langle v^1,v^2\rangle, \text{ for } i=1,2.
\end{equation}

By the equations~\eqref{exxpr} for $i=4,5$, one observes
that the vanishing of the coefficient in the term~$v^{12}$ gives the following homogeneous system
$$
\begin{pmatrix}
\lambda^1_2 & - \lambda^1_1 \\
\lambda^2_2 & - \lambda^2_1
\end{pmatrix}
\begin{pmatrix} \lambda^3_1 \\ \lambda^3_2 \end{pmatrix}
=
\begin{pmatrix} 0 \\ 0 \end{pmatrix}.
$$
By Lemma~\ref{lema1} and~\eqref{lema2-F2} we get
$\det(\lambda^i_j)_{i,j=1,2} \cdot \lambda^3_3 = \det(\lambda^i_j)_{i,j=1,2,3} \not=0$,
hence $\lambda^3_1 = \lambda^3_2 = 0$.
Consequently,
\begin{equation}\label{lema2-F3}
F(v'^{\,3})\in\langle v^3\rangle,
\end{equation}
as stated in the lemma. Now, \eqref{m3-reales4} reduces to
\begin{equation}\label{lema2-aux}
\lambda^4_4 = \lambda^{1}_1\,\lambda^3_3, \qquad \lambda^4_5 = \lambda^{1}_2\,\lambda^3_3, \qquad
\lambda^5_4 = \lambda^{2}_1\,\lambda^3_3, \qquad \lambda^5_5 = \lambda^{2}_2\,\lambda^3_3.
\end{equation}

A direct computation bearing in mind \eqref{lema2-F2} and \eqref{lema2-F3} gives
\begin{equation*}
\begin{split}
d(F(v'^{\,6})) - F(dv'^{\,6}) &= \big( \lambda^6_6  - \lambda^1_2\,\lambda^4_5 - \lambda^2_2\,\lambda^5_5 \big)\,v^{25}
+ \big( \lambda^6_7 + \lambda^3_3\,\lambda^5_4 \big)\,v^{34}
- \big( \lambda^6_6 - \lambda^3_3\,\lambda^5_5 \big)\,v^{35} + \vartheta_1,\\[2pt]
d(F(v'^{\,7})) - F(dv'^{\,7}) &= \big( \lambda^7_6  - \lambda^1_2\,\lambda^5_5 - \lambda^2_2\,\lambda^4_5 \big)\,v^{25}
+ \big( \lambda^7_7 - \lambda^3_3\,\lambda^4_4 \big)\,v^{34}
- \big( \lambda^7_6  + \lambda^3_3\,\lambda^4_5 \big)\,v^{35} + \vartheta_2,
\end{split}
\end{equation*}
where $\vartheta_1, \vartheta_2 \in \langle v^{25},v^{34},v^{35} \rangle^{\perp}  \subset  \bigwedge^2\langle v^1,\ldots,v^8\rangle$.
Therefore, the equalities \eqref{condiciones} for $i=6,7$ in particular imply that
the coefficients in $v^{34}$ and $v^{35}$ above are zero, and using~\eqref{lema2-aux} we obtain
\begin{equation}\label{m3-reales-6}
\lambda^6_6=\lambda^2_2\,(\lambda^3_3)^2, \qquad \lambda^6_7=-\lambda^2_1\,(\lambda^3_3)^2, \qquad
   \lambda^7_6=-\lambda^1_2\,(\lambda^3_3)^2, \qquad \lambda^7_7=\lambda^1_1\,(\lambda^3_3)^2.
\end{equation}
Furthermore, the coefficients in $v^{25}$ must also vanish, so taking into account~\eqref{lema2-aux}
together with the fact that $\lambda^3_3\neq 0$,
we arrive at the following equations:
$$
(\lambda^1_2)^2 + (\lambda^2_2)^2 - \lambda^2_2\,\lambda^3_3 = 0, \qquad
\lambda^1_2\,(\lambda^3_3 + 2\,\lambda^2_2) = 0.
$$
We observe that $\lambda^1_2=0$. Indeed, if $\lambda^1_2$ is non-zero, then the second expression gives
$\lambda^3_3=-2\,\lambda^2_2$ and replacing it in the first one, we arrive at
$\det(\lambda^i_j)_{i,j=1,2} =0$, which is a contradiction.
Hence, one must have $\lambda^1_2=0$, which
in turn implies $\lambda^4_5=\lambda^7_6=0$ and thus allows
us to ensure that
$$F(v'^{\,1})\in\langle v^1 \rangle, \quad \quad
F(v'^{\,4})\in\langle v^1, \ldots, v^4\rangle, \quad \text{ and } \quad
F(v'^{\,7})\in\langle v^1, \ldots, v^{5}, v^7\rangle,
$$
as desired. In particular, $\lambda^1_1\not=0$.

Finally, using~\eqref{lema2-aux} we get
\begin{equation*}
\begin{split}
d(F(v'^{\,7})) - F(dv'^{\,7}) &= - \big( \lambda^1_1\,\lambda^5_4 + \lambda^2_1\,\lambda^4_4 \big)\,v^{14}
+ \zeta
= - 2\,\lambda^1_1\,\lambda^2_1\,\lambda^3_3\,v^{14} + \zeta,
\end{split}\end{equation*}
where $\zeta \in \langle v^{14} \rangle^{\perp} \subset \bigwedge^2\langle v^1,\ldots,v^8\rangle$.
Now, the equality \eqref{condiciones} for $i=7$ implies $\lambda^2_1=0$,
and consequently, also $\lambda^5_4=\lambda^6_7=0$.
Therefore,
$$F(v'^{\,2})\in\langle v^2 \rangle,
\quad \quad
F(v'^{\,5})\in\langle v^1, v^2, v^{3}, v^5\rangle, \quad \text{ and } \quad
F(v'^{\,6})\in\langle v^1, \ldots, v^6\rangle,
$$
concluding the proof of the lemma.
\end{proof}

We next make use of the previous result to relate the values of the parameters $(\alpha,\beta)$ and $(\alpha',\beta')$.


\begin{lemma}\label{lema3}
If the nilpotent Lie algebras $\mathfrak m_{\alpha,\beta}$ and~$\mathfrak m_{\alpha',\beta'}$
are isomorphic, then
$$\alpha' = \pm \lambda^1_1\, \alpha, \qquad \beta' = \lambda^1_1\, \beta.$$
\end{lemma}

\begin{proof}
Let us first recall that all the conditions in Lemma~\ref{lema2} and in its proof must be satisfied in order to have
an isomorphism $F$ of the Lie algebras $\mathfrak m_{\alpha,\beta}$ and~$\mathfrak m_{\alpha',\beta'}$.
In particular, the expressions \eqref{lema2-aux}
and \eqref{m3-reales-6} are reduced to
\begin{equation}\label{valores-lema3}
\lambda^4_4=\lambda^1_1\,\lambda^3_3, \qquad
\lambda^5_5=\lambda^2_2\,\lambda^3_3, \qquad
\lambda^6_6=\lambda^2_2\,(\lambda^3_3)^2, \qquad
\lambda^7_7=\lambda^1_1\,(\lambda^3_3)^2.
\end{equation}
In order to prove the lemma, we need to study more deeply the conditions~\eqref{condiciones} for $i=6,7,8$
(notice that~\eqref{condiciones} are trivially fulfilled for $1\leq i\leq 5$).

Let us start with $i=6$. A direct computation applying \eqref{valores-lema3} gives us
\begin{equation*}
\begin{split}
d(F(v'^{\,6})) - F(dv'^{\,6}) &= \lambda^3_3 \left[ \lambda^2_2\,\lambda^3_3 - (\lambda^1_1)^2 \right]\,v^{14}
+ \lambda^2_2\,\lambda^3_3\,\left( \lambda^3_3 - \lambda^2_2\, \right)\,v^{25} + \xi,
\end{split}
\end{equation*}
where $\xi \in \langle v^{14},v^{25} \rangle^{\perp} \subset \bigwedge^2\langle v^1,\ldots,v^8\rangle$.
Since $\lambda^2_2\,\lambda^3_3\neq 0$, from the condition~\eqref{condiciones} for $i=6$ it follows that $\lambda^3_3=\lambda^2_2$
and $(\lambda^2_2)^2 = (\lambda^1_1)^2$, that is,
\begin{equation}\label{valores-lema3-bis}
\lambda^3_3=\lambda^2_2=\pm\lambda^1_1.
\end{equation}
Let us remark that this greatly simplifies~\eqref{valores-lema3}, which reduces to
\begin{equation}\label{valores-lema3-red}
\lambda^4_4=\pm(\lambda^1_1)^2, \qquad
\lambda^5_5=(\lambda^1_1)^2, \qquad
\lambda^6_6=\pm(\lambda^1_1)^3, \qquad
\lambda^7_7=(\lambda^1_1)^3.
\end{equation}
Furthermore,
the conditions~\eqref{condiciones} for $i=6,7$
become
\begin{equation*}
\begin{split}
0 = d(F(v'^{\,6})) - F(dv'^{\,6}) =& - \lambda^1_1\left[ \lambda^4_2 \mp \lambda^5_1 \right]\,v^{12}
+ \left[ \lambda^6_4 - \lambda^1_1\,(\lambda^4_3 \pm \lambda^5_1) \right]\,v^{13}
+ \left[ \lambda^6_5  \mp \lambda^1_1\,(\lambda^5_2 + \lambda^5_3) \right]\,v^{23}, \\[2pt]
0 = d(F(v'^{\,7})) - F(dv'^{\,7}) =& - \lambda^1_1\left[ \lambda^5_2  \mp \lambda^4_1
- \lambda^1_1\,(\alpha\,\lambda^1_1 \mp \alpha') \right]\,v^{12}
+ \left[ \lambda^7_4 - \lambda^1_1\,(\lambda^5_3  \mp \lambda^4_1) \right]\,v^{13} \\
&
+ \left[ \lambda^7_5 \pm \lambda^1_1\,(\lambda^4_2 - \lambda^4_3) \right]\,v^{23}.
\end{split}
\end{equation*}
Hence, it is easy to see that
\begin{equation}\label{valores-reducidos}
\begin{aligned}
\lambda^5_1 & =\pm\lambda^4_2, & \
    \lambda^6_4 &= \lambda^1_1\,(\lambda^4_3 + \lambda^4_2), & \
    \lambda^7_4 &= \lambda^1_1\,(\lambda^5_3  \mp \lambda^4_1),\\
 \lambda^5_2 &=\pm\lambda^4_1 + \lambda^1_1\,(\alpha\,\lambda^1_1 \mp\alpha'), & \
   \lambda^6_5 &=  \lambda^1_1\left[\lambda^4_1 \pm \lambda^5_3 - \lambda^1_1 (\alpha' \mp \alpha\,\lambda^1_1)\right], & \
   \lambda^7_5 &=  \pm\lambda^1_1\,(\lambda^4_3 - \lambda^4_2).
\end{aligned}
\end{equation}
Let us note that this makes the conditions \eqref{condiciones} to be satisfied for every $1\leq i\leq 7$, so
we must turn our attention to $i=8$.
Using~\eqref{valores-lema3-bis} and~\eqref{valores-lema3-red}, we compute
\begin{equation*}
\begin{split}
d(F(v'^{\,8})) - F(dv'^{\,8}) &= \left[ \lambda^8_6 - \lambda^1_1\,(\lambda^6_4 \pm \lambda^1_1\,\lambda^5_1) \right]\,v^{14}
+ \left[ \lambda^8_7 \mp \lambda^1_1\,( \lambda^7_4 + \lambda^1_1\,\lambda^5_2) \right]\,v^{24} \\[2pt]
&
\quad + \left[ \lambda^8_6 - 2\,\beta\,\lambda^8_8
+ \lambda^1_1\,\left( \mp \lambda^7_5 + \lambda^1_1\,(\lambda^4_2 \pm 2\,\beta'\,\lambda^1_1) \right) \right]\,v^{25}
+ \left[ \lambda^8_7  \mp \lambda^5_3\,(\lambda^1_1)^2 \right]\,v^{34} \\[2pt]
&
\quad - \left[ \lambda^8_6 + \beta\,\lambda^8_8 - (\lambda^1_1)^2\,(\lambda^4_3 \pm \beta'\,\lambda^1_1) \right]\,v^{35}
- \left[ \lambda^8_8 \mp (\lambda^1_1)^4 \right]\,v^{45} + \vartheta,
\end{split}
\end{equation*}
where $\vartheta \in \langle v^{14},v^{24},v^{25},v^{34},v^{35},v^{45} \rangle^{\perp} \subset \bigwedge^2\langle v^1,\ldots,v^8\rangle$.
Equalling to zero the coefficients in $v^{14}$, $v^{34}$, and~$v^{45}$, and bearing in mind
\eqref{valores-reducidos},
one has
$$
\lambda^8_6 = (\lambda^1_1)^2\,(\lambda^4_3 + 2\,\lambda^4_2), \qquad
\lambda^8_7 = \pm \lambda^5_3\,(\lambda^1_1)^2, \qquad
\lambda^8_8=\pm (\lambda^1_1)^4.
$$
Moreover, from these equalities together with the vanishing of the coefficient in $v^{35}$, it is easy to see that
$$\lambda^4_2 = \pm \frac{\lambda^1_1}{2}\,(  \beta' - \beta\,\lambda^1_1).$$
We now replace all these values in the coefficients of $v^{24}$ and $v^{25}$. Their annihilation gives us the
desired relation between $(\alpha,\beta)$ and $(\alpha',\beta')$.
\end{proof}

As we observed above, Lemma~\ref{lema3} implies Proposition~\ref{teorema-no-iso-NLA-bis} with $\rho=\lambda^1_1\not= 0$,
because $a,b,a',b'$ are related to
$\alpha,\beta,\alpha',\beta'$ by $\alpha = \frac{4\,a}{\sqrt3}$,
$\beta= \frac{2\,b}{3}$, $\alpha' = \frac{4\,a'}{\sqrt3}$ and $\beta'= \frac{2\,b'}{3}$.

\section{An infinite family of nilmanifolds $N_{a}$ with complex structures}\label{nilvariedades}

\noindent In this section we show how the results of the previous section
allow us to conclude that in dimension eight
there exist infinitely many real homotopy types of nilmanifolds admitting complex structures.

Let us start reviewing some results about minimal models and homotopy theory of nilmanifolds.
Sullivan showed in \cite{Sullivan} that it is possible to associate to any nilpotent CW-complex $X$ a
minimal model, i.e. a commutative differential graded algebra, cdga, $(\bigwedge V_X,d)$ defined over the rational numbers $\mathbb{Q}$
satisfying a certain minimality condition,
which encodes the rational homotopy type of~$X$~\cite{GM-libro}.

Recall that a space $X$ is nilpotent
if its fundamental group $\pi_1(X)$ is a nilpotent group and acts in a nilpotent way on the
higher homotopy groups $\pi_k(X)$ for $k > 1$. These conditions are fulfilled when the space is a nilmanifold $N$,
i.e. $N=\nilm$ is a compact quotient of a connected, simply connected,
nilpotent Lie group $G$ by a lattice $\Gamma$ of maximal rank.
Indeed, for any nilmanifold $N$ one has $\pi_1(N)=\Gamma$, hence nilpotent, and $\pi_k(N)=0$ for every $k \geq 2$.

A commutative differential graded algebra $(\bigwedge V,d)$ defined over
$\mathbb{K}\,(=\mathbb{Q}$ or $\mathbb{R})$ is said to be minimal if the following conditions are satisfied:
\begin{enumerate}
\item[\emph{(1)}] $\bigwedge V$ is the free commutative algebra generated by the graded vector
space $V = \oplus V^k$;
\item[\emph{(2)}] there exists a basis $\{x_j\}_{j \in J}$, for some well-ordered index set $J$, such
that $\deg(x_i) \leq \deg(x_j)$ if $i < j$, and each $dx_j$ is expressed in terms of
preceding $x_i$ $(i < j)$.
\end{enumerate}

Given a differentiable manifold $M$, a $\mathbb{K}$-minimal model of $M$ is a minimal cdga $(\bigwedge V,d)$ over~$\mathbb{K}$
together with a quasi-isomorphism $\phi$ from $(\bigwedge V,d)$ to the $\mathbb{K}$-de Rham complex of $M$,
i.e. a morphism~$\phi$ inducing an isomorphism on cohomology.
Note that the $\mathbb{K}$-de Rham complex of $M$ is the usual de Rham complex of differential forms when $\mathbb{K}=\mathbb{R}$,
whereas for $\mathbb{K}=\mathbb{Q}$ one considers the $\mathbb{Q}$-polynomial forms instead.
Two manifolds $M_1$ and $M_2$ have the same $\mathbb{K}$-homotopy type
if and only if their $\mathbb{K}$-minimal models are isomorphic \cite{DGMS,Sullivan}.
Clearly, if
$M_1$ and $M_2$ have different real homotopy types, then $M_1$ and $M_2$ also have different rational homotopy types.

Let $N=\nilm$ be a nilmanifold of dimension $n$ and denote by $\frg$ the Lie algebra of $G$.
The minimal model of $N=\nilm$ is given by the Chevalley-Eilenberg complex $(\bigwedge\frg^*,d)$
of the nilpotent Lie algebra $\frg$. Indeed, according to Mal'cev \cite{Malcev},
the existence of the lattice $\Gamma$ of maximal rank in $G$ is equivalent to
the nilpotent Lie algebra $\frg$ being rational. The latter condition is in turn equivalent to
the existence of a basis $\{e^1,\ldots,e^n\}$ for the dual $\frg^*$ for which
the structure constants are rational numbers.
Since the Lie algebra~$\frg$
is nilpotent, one can take
a basis $\{e^1,\ldots,e^n\}$ for $\frg^*$ as above satisfying
$$
de^1=de^2=0, \quad de^j=\sum_{i,k<j} a^j_{ik}\, e^i\wedge e^k \ \ \mbox{ for } j=3,\ldots, n,
$$
where $a^j_{ik} \in \mathbb{Q}$.
Therefore, the cdga $(\bigwedge\frg^*,d)$ over $\mathbb{Q}$ is minimal, because it satisfies \emph{(1)} and \emph{(2)}
just taking $J=\{1,\ldots,n\}$
and $V=V^1=\langle x_1,\ldots,x_n \rangle= \sum_{j=1}^n \mathbb{Q}x_j $, where $x_j=e^j$ for $1 \leq j \leq n$.
That is to say, there are $n$ generators $x_1,\ldots,x_n$ of degree $1$, and one has the $\mathbb{Q}$-minimal cdga
\begin{equation}\label{Q-minimal}
\big(\bigwedge\langle x_1,\ldots,x_n\rangle,\,d\,\big).
\end{equation}

Now, the cdga $(\bigwedge\frg^*,d)$ over $\mathbb{R}$ is also minimal, since
it is given by
\begin{equation}\label{R-minimal}
\big(\bigwedge\langle x_1,\ldots,x_n\rangle \otimes \mathbb{R},\,d\,\big).
\end{equation}
Nomizu proved in \cite{Nomizu} that the canonical morphism $\phi$ from the Chevalley-Eilenberg complex $(\bigwedge\frg^*,d)$
to the de Rham complex $(\Omega^*(\nilm),d)$ induces an isomorphism on cohomology. Therefore, the $\mathbb{R}$-minimal model of the nilmanifold
$N=\nilm$ is given by \eqref{R-minimal}.
As it was observed by Hasegawa in~\cite{Has-PAMS},
the cdga \eqref{Q-minimal} is also the $\mathbb{Q}$-minimal model of $N$. Moreover, by using a result
in \cite{DGMS} asserting that a $\mathbb{K}$-minimal model, $\mathbb{K}=\mathbb{Q}$ or $\mathbb{R}$, of a compact birational K\"ahler manifold
is formal, Hasegawa proved that an even-dimensional nilmanifold does not admit any K\"ahler metric unless it is a torus,
by showing that the minimal model \eqref{Q-minimal} is never formal.
(For more results on rational homotopy theory see~\cite{FHT-libro},
and for more applications to symplectic geometry see~\cite{OT-libro}.)

It is proved in \cite{BM2012} that, up to dimension $5$, the number of rational homotopy types of nilmanifolds is finite.
However, in six dimensions the following result holds:

\begin{theorem}\label{teorema-BM}\cite[Theorem 2]{BM2012}
There are $34$ real homotopy types of $6$-dimensional nilmanifolds, and infinitely many rational
homotopy types of $6$-dimensional nilmanifolds.
\end{theorem}

As an obvious consequence, there is a finite number of real homotopy types of $6$-dimensional nilmanifolds
admitting a complex structure. In what follows, we will show that for nilmanifolds with complex structures in eight dimensions
there are infinitely many real homotopy types. To see this, let us take a non-negative rational number $a$
and consider the connected, simply connected,
nilpotent Lie group $G_a$ corresponding to the nilpotent Lie algebra $\frg_{a,1}$ given in Section~\ref{sect3}.
It follows from~\eqref{ec-reales-FII} that the algebra $\frg_{a,1}$ is rational, hence by Mal'cev \cite{Malcev},
there exists a lattice $\Gamma_a$ of maximal rank in $G_a$. We denote by $N_a=\Gamma_a\backslash G_a$ the
corresponding compact quotient. The nilmanifold $N_a$ admits complex structures, for instance,
the strongly non-nilpotent complex structure $J_{a,1}$ defined in~\eqref{J-a,b} for $b=1$.

Now, as we explained above, the $\mathbb{R}$-minimal model of the nilmanifold
$N_a$ is given by the Chevalley-Eilenberg complex $(\bigwedge\frg_{a,1}^*,d)$
of the nilpotent Lie algebra $\frg_{a,1}$.
As a consequence of Theorem~\ref{teorema-main}, one has that two such $\mathbb{R}$-minimal models are isomorphic
if and only if $a=a'$.
Therefore, we have proved the following result:

\begin{theorem}\label{infinite}
There are infinitely many real homotopy types of $8$-dimensional nilmanifolds admitting a complex structure.
\end{theorem}

\medskip

A direct calculation using Nomizu's theorem \cite{Nomizu} allows to explicitly compute the de Rham cohomology groups of any nilmanifold $N_{a}$.
In our case, we have:
\begin{equation*}
\begin{split}
H_{\text{dR}}^1(N_{a})=\langle\, &[e^1],\,[e^2],\,[e^3] \,\rangle, \\[2pt]
H_{\text{dR}}^2(N_{a})=\langle\, &[e^{12}],\,[e^{14}],\,[e^{25}],\,[e^{34}] \,\rangle, \\[2pt]
H_{\text{dR}}^3(N_{a})= \langle\, &[e^{147}],\,[e^{257}\!-2\,e^{346}],\,
  [3\,e^{146}\!-e^{256}\!+2\,e^{356}],\,[e^{156}\!+e^{246}\!-2\,e^{346}], \\
&
[3\,e^{157}\!+3\,e^{247}\!+2\,e^{256}\!+2\,e^{356}],\,[2\, e^{126}\!+3\,e^{128}\!+4\,e^{237}\!-2\,e^{256}],
\,[2 a\,e^{237}\!-e^{346}\!-e^{357}] \,\rangle.
\end{split}
\end{equation*}

By Poincar\'e duality, the following Betti numbers for $N_a$ are obtained:
$$
b_0(N_a)=b_8(N_a)=1, \quad b_1(N_a)=b_7(N_a)=3, \quad b_2(N_a)=b_6(N_a)=4, \quad b_3(N_a)=b_5(N_a)=7.
$$
The Betti number $b_4(N_a)$ can be computed by taking into account that the Euler characteristic $\chi$ of
a nilmanifold always vanishes. Hence,
$$
0=\chi(N_a) = \sum_{k=0}^8 (-1)^k b_k(N_a) = b_4(N_a) + 2\left( b_0(N_a)-b_1(N_a)+b_2(N_a)-b_3(N_a) \right),
$$
which implies $b_4(N_a)=10$.

\smallskip

The second de Rham cohomology group provides the following additional geometric information on the nilmanifolds $N_a$.

\begin{proposition}\label{Na-no-simplectica}
For each non-negative rational number $a$,
the nilmanifold $N_a$ does not admit any symplectic structure.
\end{proposition}

\begin{proof}
It is enough to see that any de Rham cohomology class $\mathfrak{a}$ in $H_{\text{dR}}^2(N_{a})$ is degenerate.
Indeed, since $\mathfrak{a}=\lambda_1[e^{12}]+\lambda_2 [e^{14}]+\lambda_3[e^{25}]+\lambda_4[e^{34}]$
it is clear that the cup product $\mathfrak{a}\cup\mathfrak{a}\cup\mathfrak{a}$ vanishes, so $\mathfrak{a}^4\equiv 0$
and $\mathfrak{a}$ is degenerate.
\end{proof}

In contrast to this result, in the following section we will show that
these nilmanifolds have many special Hermitian metrics.

\section{Special metrics on the nilmanifolds $N_{a}$}\label{Hmetrics}

\noindent
We here study the existence of special Hermitian metrics on the nilmanifolds $N_{a}$ endowed with
the strongly non-nilpotent complex structures given in the previous sections.
More concretely, we find many generalized Gauduchon metrics as well as balanced Hermitian metrics.

Let us start by recalling the definition and the main properties of the $k$-th Gauduchon metrics studied in \cite{FWW}.

\begin{definition} \cite{FWW}
{\rm
Let  $X=(M, J)$ be a compact complex manifold  of complex dimension $n$, and let~$k$ be an integer such that
$1 \leq k \leq n -1$.
A $J$-Hermitian metric $g$ on $X$ is called \emph{$k$-th Gauduchon} if its fundamental $2$-form $F$ satisfies the condition
$\partial \db F^k  \wedge F^{n -k-1} =0$.
}
\end{definition}

Clearly, any SKT metric (a $J$-Hermitian metric $g$ with fundamental form $F$ satisfying $\partial \db F=0$) is a $1$-st Gauduchon metric,
and any astheno-K\"ahler metric (i.e. a $J$-Hermitian metric $g$ such that its fundamental form $F$ satisfies the condition $\partial \db F^{n-2}=0$) is an $(n-2)$-th Gauduchon metric.
SKT, resp. astheno-K\"ahler, metrics were first introduced by Bismut in~\cite{Bismut}, resp. by Jost and Yau in \cite{JY},
and they have been further studied by many authors.
Notice that for $k = n-1$ one gets the standard metrics found in \cite{Gau}.

In \cite{FWW} an extension of Gauduchon's result for standard metrics is proved.
More concretely, it is shown in \cite[Corollary 4]{FWW} that if $(M,J,F)$ is an $n$-dimensional compact Hermitian manifold,
then for any integer $1 \leq k \leq n - 1$,
there exists a unique constant $\gamma_k (F)$ and a (unique up to a constant) function $v \in {\mathcal C}^{\infty} (M)$ such that
$$
\frac{i}{2} \partial \db (e^v F^k) \wedge F^{n - k -1} = \gamma_k (F) e^v F^n.
$$
If $(M,J,F)$ is K\"ahler, then $\gamma_k(F) =0$ and $v$ is a constant function for any $1 \leq k \leq n - 1$.
Moreover, for any Hermitian metric on a compact complex manifold one has $\gamma_{n-1}(F) =0$.

The constant $\gamma_k(F)$ depends smoothly on $F$ (see \cite[Proposition 9]{FWW}).
Furthermore, it is proved in \cite[Proposition~8]{FWW}
that $\gamma_k (F) = 0$ if and only if there exists a $k$-th Gauduchon metric in the conformal class of~$F$.

Some compact complex manifolds with generalized Gauduchon metrics
are constructed in \cite{FU,FWW,IP,LUV-DGA} by different methods.
In the following result we find many generalized Gauduchon nilmanifolds
of complex dimension 4.

As in Section~\ref{nilvariedades}, for each non-negative rational number $a$, we consider the
nilmanifold $N_{a}$ endowed with the complex structure given by~\eqref{J-a,b} for $b=1$, i.e. with the complex structure $J_{a,1}$.
Hence, $(N_a,J_{a,1})$ is a compact complex manifold of complex dimension 4.
We will denote this complex nilmanifold by $X_a$, i.e. $X_a=(N_a,J_{a,1})$.

\begin{theorem}\label{infinite-1-Gauduchon}
For each non-negative rational number $a$, the complex nilmanifold $X_a$ has a $k$-th Gauduchon metric, for every $1 \leq k \leq 3$.
Therefore, there are infinitely many real homotopy types of $8$-dimensional $k$-th Gauduchon nilmanifolds, for every $1 \leq k \leq 3$.
\end{theorem}

\begin{proof}
By Theorem~\ref{infinite}, we know that the nilmanifolds $N_a$ and $N_{a'}$ have different real homotopy types for $a \not= a'$.
The complex structure equations for $X_a=(N_a,J_{a,1})$ are given by~\eqref{eleccion} with $b=1$, i.e.
\begin{equation}\label{eleccion-bis}
\left\{\begin{array}{rcl}
d\omega^1 &=& 0,\\[2pt]
d\omega^2 &=& -\omega^{14} + \omega^{1\bar 4},\\[2pt]
d\omega^3 &=& a\,\omega^{1\bar 1} - i\,(\omega^{12} - \omega^{1\bar 2} - \omega^{2\bar1})
  + i\,(\omega^{24}-\omega^{2\bar 4}),\\[2pt]
d\omega^4 &=& i\,(\omega^{1\bar 2} + \omega^{2\bar 1}) + \omega^{1\bar 3}
  +i\,\omega^{2\bar 2}-\omega^{3\bar 1}.
  \end{array}\right.
\end{equation}

We define the $(1,1)$-form $F$ on $X_{a}$ given by
$$
F =
\frac{101}{4} \,i\, \omega^{1\bar{1}} + i\,\omega^{2\bar{2}} + i\,\omega^{3\bar{3}} + \frac{1}{2} i\, \omega^{4\bar{4}}
- 5\,\omega^{1\bar{3}} + 5\, \omega^{3\bar{1}}.
$$
It is easy to check that $F$ defines a positive-definite metric on $X_a$, hence a Hermitian metric on~$X_a$.
Furthermore, a direct calculation using the structure equations~\eqref{eleccion-bis} shows that
\begin{equation*}
\begin{split}
\partial\db F =& -4\,i\,\omega^{12\bar 1\bar 2} - (1-6\,i)\,\omega^{12\bar 1 \bar 3} + a\,\omega^{12\bar 1\bar 4}
   - (1-i)\,\omega^{12\bar2\bar3} + (1+6\,i)\,\omega^{13\bar 1\bar2}
   - i\,\omega^{13\bar 1\bar3} \\
   &- \omega^{13\bar 2\bar3} - a\,\omega^{14\bar 1\bar2} - 12\,i\,\omega^{14\bar 1\bar4} + \omega^{14\bar 3\bar4}
   + (1+i)\,\omega^{23\bar 1\bar2} + \omega^{23\bar 1\bar 3} - 2\,i\,\omega^{24\bar 2\bar 4}
   - \omega^{34\bar 1\bar 4}.
\end{split}
\end{equation*}
Therefore, one has $\partial\db F\wedge F^2=0$
for any $a$, i.e. $F$ is a $1$-st Gauduchon metric on the
complex nilmanifold $X_a$ for any $a$.

By \cite[Proposition 2.2]{LU}, if an invariant Hermitian
metric $F$ on a complex nilmanifold $X$ of complex dimension $n\geq 4$
is $k$-th Gauduchon for some $1 \leq k \leq n-2$, then it is $k$-th Gauduchon for any other $k$.
Since any invariant Hermitian metric is $(n-1)$-Gauduchon, we conclude that
the $1$-st Gauduchon metric on $X_a$ constructed above
are also $k$-th Gauduchon for $k=2,3$.
\end{proof}

It can be seen that the compact complex manifolds $X_{a}$ are not SKT (see \cite{EFV} for general results).
Moreover, it can be proved that the manifolds $X_{a}$ do not admit any {\it invariant} astheno-K\"ahler metric.

\medskip

A particularly interesting class of standard metrics is the one given by \emph{balanced} Hermitian metrics,
defined by the condition $dF^{n-1}=0$. Important aspects of these metrics were first investigated by Michelshon in~\cite{Mi},
and many authors have constructed balanced manifolds and studied their properties since then.
In the following result we find many balanced nilmanifolds of complex dimension~4.

\begin{theorem}\label{infinite-balanced}
For each non-negative rational number $a$, the complex nilmanifold $X_a$ has a balanced Hermitian metric.
Therefore, there are infinitely many real homotopy types of $8$-dimensional balanced nilmanifolds.
\end{theorem}

\begin{proof}
We will define a balanced Hermitian metric on $X_a=(N_a,J_{a,1})$ depending on the values of the rational number $a$.

For $a=0$, we consider the $(1,1)$-form $F$ on $X_{0}$ given by
$$
2\,F_0 = i\, (2\,\omega^{1\bar{1}} + \omega^{2\bar{2}} + 4\,\omega^{3\bar{3}} + \omega^{4\bar{4}}) + 2\,\omega^{1\bar{3}} - 2\,\omega^{3\bar{1}}.
$$
Hence, $F_0$ defines a positive-definite metric on $X_0$.
Since $F_0$ is real, the closedness of $F_0^3$ is equivalent to the condition
$$
\partial F_0 \wedge F_0^2 = 0.
$$
Using the structure equations~\eqref{eleccion-bis} for $a=0$, we get
\begin{equation*}
\begin{array}{rl}
2\, \partial F_0 =\!\!&\!\! 4\,\omega^{12\bar{3}} - 4\,\omega^{13\bar{2}} +(i-1)\omega^{14\bar{2}} + i\,\omega^{14\bar{3}}
- 4\,\omega^{23\bar{1}}  \\[4pt]
\!\!&\!\! -(i+1)\omega^{24\bar{1}} -\omega^{24\bar{2}} - 4\,\omega^{24\bar{3}} -i\,\omega^{34\bar{1}} - 4\,\omega^{34\bar{2}}.
\end{array}
\end{equation*}
One can check that $\partial F_0 \wedge F_0^2 = 0$, i.e. $d F_0^3=0$,
so $F_0$ is a balanced Hermitian metric on the
complex nilmanifold $X_0$.

For any rational number $a> 0$, we define the $(1,1)$-form $F_a$ on $X_{a}$ given by
$$
2\,F_a = i \left( a(a+1) \omega^{1\bar{1}} + \omega^{2\bar{2}} + \omega^{3\bar{3}} + 2\,\omega^{4\bar{4}} \right)
+ a\,\omega^{1\bar{3}} - a\,\omega^{3\bar{1}}
+ \omega^{2\bar{4}} - \omega^{4\bar{2}}.
$$
It is clear that $F_a$ defines a Hermitian metric on $X_a$ $(a> 0)$.
A direct calculation using the structure equations~\eqref{eleccion-bis} shows that
\begin{equation*}
\begin{array}{rl}
2\, \partial F_a =\!\!&\!\! i\,\omega^{12\bar{2}} +2\,\omega^{12\bar{3}} - ai\,\omega^{13\bar{1}} - \omega^{13\bar{2}}
-\left(2-i(a-1)\right) \omega^{14\bar{2}} + 2i\,\omega^{14\bar{3}} \\[4pt]
\!\!&\!\! - \omega^{14\bar{4}}  -\left(2+i(a-1)\right) \omega^{24\bar{1}} -2\,\omega^{24\bar{2}} - \omega^{24\bar{3}} -2i\,\omega^{34\bar{1}} - \omega^{34\bar{2}}.
\end{array}
\end{equation*}
Moreover, one can check that
$\partial F_a \wedge F_a^2 = 0$, and consequently
$d F_a^3=0$. Therefore, $F_a$ is a balanced metric on the
complex nilmanifold $X_a$ for any positive rational number $a$.
\end{proof}

In~\cite{Pop0,Pop1} Popovici has introduced and investigated the class of \emph{strongly Gauduchon} metrics $F$,
which are defined by the condition
$\partial F^{n-1}=\db\gamma$, for some complex form $\gamma$ of bidegree $(n,n-2)$ on a compact complex manifold $X$
of complex dimension $n$.
It is clear by definition that any balanced metric is strongly Gauduchon, and any strongly Gauduchon metric is standard.

Hence, as an immediate consequence of Theorem~\ref{infinite-balanced} we get:

\begin{corollary}\label{infinite-sG}
There are infinitely many real homotopy types of $8$-dimensional strongly Gauduchon nilmanifolds.
\end{corollary}

It is well known that the existence of a K\"ahler metric on an even dimensional nilmanifold reduces the nilmanifold
to be a torus, so the underlying Lie algebra is abelian and the minimal model is formal~\cite{Has-PAMS}.
More generally, it seems natural to ask if the existence of other geometric structures on nilmanifolds
imposes some restriction on their real homotopy type. For instance, as mentioned in the introduction,
it is an interesting question if the number of real homotopy types of $8$-dimensional nilmanifolds
admitting an invariant hypercomplex structure is finite or not.
In addition to this problem, the answers to the following questions are, to our knowledge, unknown:

\medskip

\noindent\textbf{Question A.} Are there infinitely many real homotopy types of $8$-dimensional SKT nilmanifolds?

\medskip

\noindent\textbf{Question B.} Are there infinitely many real homotopy types of $8$-dimensional astheno-K\"ahler

\hskip1.8cm nilmanifolds?




\end{document}